\documentclass[11pt, reqno]{amsart}
\usepackage{amssymb,latexsym,amsmath,amsfonts}
\usepackage{latexsym}
\usepackage{latexsym}
\usepackage{hyperref}
\usepackage[mathscr]{eucal}
\usepackage[english]{babel}
\usepackage{fullpage}

\newtheorem*{thm1}{ \bf Theorem 1}
\newtheorem*{thm2}{ \bf Theorem 2}

\newtheorem{lem}{ \bf Lemma}[section]
\newtheorem{defn}{ \bf Definition}[section]

\newtheorem{prop}{\bf Proposition}[section]

\newcommand{\be}{\begin{equation}}
\newcommand{\ee}{\end{equation}}
\newcommand{\Bea}{\begin{eqnarray*}}
\newcommand{\Eea}{\end{eqnarray*}}
\newcommand{\bea}{\begin{eqnarray}}
\newcommand{\eea}{\end{eqnarray}}

\numberwithin{equation}{section}

\def\dd{{\delta}^D}
\def\Rt{{\check{R}}}
\def\gt{\tilde{g}}
\def\la{\Delta}
\def\a{\alpha}

\begin{document}
\title[Unstable critical metrics]{Some Unstable critical metrics for the $L^{\frac{n}{2}}$-norm of\\ the curvature tensor }
\author{Atreyee Bhattacharya and Soma Maity}
\address{Department of Mathematics, Indian Institute of Science,
Bangalore-12, India}
\email{somamaity@math.iisc.ernet.in\\atreyee@math.iisc.ernet.in}
\begin{abstract}
We consider the Riemannian functional defined on the space of Riemannian metrics with unit volume on a closed smooth manifold $M$ given by $\mathcal{R}_{\frac{n}{2}}(g):= \int_M |R(g)|^{\frac{n}{2}}dv_g$ where $R(g)$, $dv_g$ denote the Riemannian curvature and volume form corresponding to $g$. We show that there are locally symmetric spaces which are unstable critical points for this functional.
\end{abstract}
\thanks{}
\maketitle
\section{Introduction} Let $M$ be a closed smooth manifold of dimension $n\geq 3$ and $\mathcal{M}_1$ be the space of Riemannian metrics with unit volume on $M$ endowed with the $C^{2,\a}$-topology for any $\a\in (0,1)$. In this paper we study the following Riemannian functional
\Bea \mathcal{R}_p(g)=\int_M|R(g)|^pdv_g
\Eea
where $R(g)$ and $dv_g$ denote the corresponding Riemannian curvature tensor and volume form, $p\in [2,\infty).$ Let $S^2(T^*M)$ be the space of symmetric two tensors on $M$ and $\mathcal{W}$ be the subspace of $S^2(T^*M)$ orthogonal to the tangent space of the orbit of $g$ under the action of the group of diffeomorphisms of $M$ at $g$. Let $H$ denote the Hessian of $\mathcal{R}_p$ at any critical metric. For definitions of critical metric and Hessian we refer to section 2.
\begin{defn}
{\rm  Let $g$ be a critical point for $\mathcal{R}_{p|\mathcal{M}_1}$. $g$ is {\it stable} for $\mathcal{R}_p$ if there is an $\epsilon>0$ such that for every element $h$ in $\mathcal{W}$,
\be H(h,h)\geq \epsilon \|h\|^2
\ee
where $\|.\|$ denote the $L^2$-norm on $S^2(T^*M)$ defined by $g$}.
\end{defn}
Spherical space forms are stable for $\mathcal{R}_p$ for $p\geq 2$ and hyperbolic manifolds are stable for $p\geq \frac{n}{2}$ \cite{SM}. The stability of locally symmetric spaces is not known in general. In this paper we prove the following.
\begin{thm1}  Let $(M,g)$ be an irreducible locally symmetric space of compact type. If the universal cover of $M$ is one of the following then $(M,g)$ is not stable for $\mathcal{R}_{\frac{n}{2}}$.
$$ SU(q) (q\geq 3), \ \ Sp(q)(q\geq 2), \ \ Spin(5), \ \ Spin(6), \ \ SU(2q+2)/Sp(q+1),$$
$$Sp(q+l)/Sp(q)\times Sp(l)(l,q \geq 1), \ \ E_6/F_4, \ \  F_4/Spin(9)$$
Moreover, $(M,g)$ is a saddle point for $\mathcal{R}_{\frac{n}{2}}.$
\end{thm1}
The theorem follows by restricting $H$ to the space of conformal variations of any irreducible symmetric space and using an estimate for the first positive eigenvalue of the Laplacian of $(M,g)$ . \\
\\
Let $(M,g)$ be a simply connected irreducible symmetric space of compact and $\lambda_1$ and $s$ denote its first positive eigenvalue of the Laplacian and scalar curvature of it. We prove that if $\frac{\lambda_1}{s}\geq \frac{2}{n}$ then $(M,g)$ is stable for $\mathcal{R}_{\frac{n}{2}}$ restricted to the conformal variations of $g$. The above condition is also a necessary and sufficient criterion for the stability of the identity map of $(M,g)$ as a harmonic map. In \cite{HU} the stability of the identity map of these spaces has been studied in detail. We observe that if $(M,g)$ is not a sphere then $g$ is stable for $\mathcal{R}_{\frac{n}{2}}$ if and only if it is stable for the identity map.\\
\\
Let $(M,g)$ be an irreducible symmetric space of compact type or a compact quotient of an irreducible locally symmetric space of non-compact type. From the proof of the theorem we observe that if $(M,g)$ is neither one of the type in Theorem 1 then $(M,g)$ is stable for $\mathcal{R}_p (p\geq \frac{n}{2})$ restricted to the conformal variations of $(M,g)$.
\section{Proof}
Let $\{e_i\}$ be an orthonormal basis at a point of $M$. $\Rt$ is a symmetric 2-tensor defined by $$\Rt(x,y)=\sum R(x,e_i,e_j,e_k)R(y,e_i,e_j,e_k).$$
Let $D$ and $D^*$ be the Riemannian connection, its formal adjoint  and $s$ denote the scalar curvature.
\\
$d^D:S^2(T^*M)\to \Gamma (T^*M\otimes\Lambda^2M)$ and its formal adjoint $\dd$ are defined by
\Bea &&d^D\alpha(x,y,z):= (D_y\alpha)(x,z)-(D_z\alpha)(x,y)\\
&&\dd(A)(x,y)=\sum\{D_{e_i}A(x,y,e_i) +D_{e_i}A(y,x,e_i)\}
\Eea
where $\Lambda^2M$ and $\Gamma (T^*M\otimes\Lambda^2M)$ denote alternating two forms and sections of $T^*M\otimes\Lambda^2M$.
Let $g_t$ be a one-parameter family of metrics with $\frac{d}{dt}(g_t)_{|t=0}=h$ and $T(t)$ be a tensor depending on $g_t.$ Then $\frac{d}{dt}T(t)_{|t=0}$ is denoted by $T'_g(h)$. Define $\Pi_h(x,y)= \frac{d}{ dt}D_xy_{|t=0}$ where $x,y$ are two fixed vector fields. The suffix $h$ will be omitted when there will not be any ambiguity.
Consider any $f\in C^{\infty}(M)$. Note that
\bea g(\Pi_{fg}(x,y),z)&=& \frac{1}{2}[D_xfg(y,z)+D_yfg(x,z)-D_zfg(x,y)]\\
\nonumber&=& \frac{1}{2}[df(x)g(y,z)+df(y)g(x,z)-df(z)g(x,y)]
\eea
Let $\la$ denote the Laplace operator which acts on $C^{\infty}(M).$ We use the following definition.
$$\la f=-tr(Ddf)$$
$(,)$, $|.|$, $\langle,\rangle$, $\|.\|$ denote point-wise inner product, point-wise norm, global inner product and global norms induced by $g.$

$\nabla \mathcal{R}_p(g)$ in $S^2(T^*M)$ is called the {\it gradient} of $\mathcal{R}_p$ at $g$ if for every $h\in S^2(T^*M)$,
$$\frac{d}{dt}_{|t=0}\mathcal{R}_p(g+th)= \mathcal{R}_{p|g}'.h=\langle \nabla \mathcal{R}_p(g),h\rangle$$
$g$ is called a {\it critical point} for $\mathcal{R}_{p|\mathcal{M}_1}$ if the component of $\nabla\mathcal{R}_p(g)$ along the tangent space of $\mathcal{M}_1$ at $g$ is zero. The Hessian at a critical point of $\mathcal{R}_p$ is given by
 $$H(h_1,h_2)=\langle (\nabla \mathcal{R}_p)'_g(h_1),h_2 \rangle \hspace{3mm} \forall \ h_1,h_2\in S^2(T^*M)$$

\begin{prop} Let $(M,g)$ be a compact irreducible symmetric space and $f\in C^{\infty}(M)$. Then
\Bea H(fg,fg)=p|R|^{p-2}[a\|\la f\|^2-b\|df\|^2+c\|f\|^2]
\Eea
where, $a$, $b$, $c$ is given by,
\Bea &&a= n-1+(p-2)\frac{4s^2}{n^2|R|^2}\\
     &&b=4(p-1)\frac{s}{n}\\
     && c=(p-\frac{n}{2})|R|^2.
\Eea
\end{prop}
Let $(M,g)$ be a closed irreducible symmetric space and $h_1,h_2\in S^2(T^*M).$ From \cite{SM} (4.1) we have,
\bea H(h_1, h_2) &=& -p|R|^{p-2}(\langle \dd (D^*)'_g(h_1)R,h_2\rangle+\langle D^*R'_g(h_1),d^D h_2\rangle)
-p|R|^{p-2}\langle\Rt '_g(h_1),h_2\rangle\\
\nonumber&&-p\langle(|R|^{p-2})'_g(h_1)R, Dd^D h_2\rangle-\frac{p}{n}|R|^2\langle(|R|^{p-2})'_g(h_1)g,h_2\rangle\\
\nonumber&&+\frac{1}{2}\langle(|R|^{p})'_g(h_1)g,h_2\rangle
+\frac{p}{n}\|R\|^p\langle h_1,h_2\rangle
\eea
Next we compute each term of the above equation for conformal variations to obtain the Proposition.

\begin{lem}$\langle \dd(D^*)'(fg)R,fg\rangle=4\frac{s}{n}\|df\|^2 $
\end{lem}
\begin{proof} Let $\gt(t)$ be an one-parameter family of metrics with $\tilde{g}(0)=g$ and $T$ be a $(0,4)$ tensor independent of $t$. Expressing $D^*$ in a local coordinate chart and differentiating it, we obtain,
\Bea (D^*T)(x,y,z)=-(\gt^{kj})'(D_kT)_{jxyz}+\gt^{kj}[T_{\Pi_{kj}xyz}
+T_{j\Pi_{kx}yz}+T_{jx\Pi_{ky}z}+T_{jxy\Pi_{kz}}]
\Eea
Note that, $\Pi$ acting on two vector fields gives a vector field. Now evaluating $(D^*)'_g(h)(R)$ on an orthonormal basis we have,
\Bea (D^*)'_g(h)(R)_{jkl}= R_{\Pi_{ii}jkl}+R_{i\Pi_{ij}kl}+R_{ij\Pi_{ik}l}+R_{ijk\Pi_{il}}.
\Eea
From the definition of $d^D$ we have,
$$d^Dfg(x,y,z)=D_yfg(x,z)-D_zfg(x,y)=df(y)g(x,z)-df(z)g(x,y)$$
Combining these two,
\Bea \sum(D^*)'(fg)_{jkl}d^D(fg)_{jkl}&=&2\sum[R_{\Pi_{ii}jkj}+R_{i\Pi_{ij}kj}+R_{ij\Pi_{ik}j}]df_k\\
\Eea
Let $\mu$ be the Einstein constant of $(M,g).$ Now using (2.2) we have,
\bea 2\sum R_{\Pi_{ii}jkj}df_k&=&2\sum g(\Pi(e_i,e_i),e_m)R(e_m,e_j,e_k,e_j)df(e_k)\\
\nonumber&=&-(n-2)\mu |df|^2
\eea
Similarly,
\be 2\sum R_{ij\Pi_{ik}j}df_k=n\mu|df|^2 \ \ {\rm and} \ \  2\sum R_{i\Pi_{ij}kj}df_k=2\mu|df|^2
\ee
Combining the equations (2.3) and (2.4) the proof of the lemma follows.
\end{proof}
\begin{lem}$\langle D^*R'_g(fg),d^D fg\rangle=-(n-1)\|\la f\|^2+(n-4)\frac{s}{n}\|df\|^2$
\end{lem}
\begin{proof} By a simple calculation we have,
\be D^2_{x,y}fg(u,v)=Ddf(x,y)g(u,v)
\ee
and
\be Dd^Dfg(x,y,z,w)=Ddf(x,z)g(y,w)-Ddf(x,w)g(y,z).\ee
From \cite{BA} 1.174(c) and using (2.5) we have,
\bea R'_g(fg)(x,y,z,u)&=&-\frac{1}{2}[Ddf(y,z)g(x,u)+Ddf(x,u)g(y,z)-Ddf(x,z)g(y,u)\\
\nonumber &&-Ddf(y,u)g(x,z)]+fR(x,y,z,u)\eea
Therefore,
\Bea (R'_g(fg),Dd^Dfg)&=&Ddf_{ik}R'_g(fg)_{ijkj}-Ddf_{il}R'_g(fg)_{ijjl}\\
&=& -(n-2)|Ddf|^2-|\la f|^2-2\mu|df|^2
\Eea
Using Bochner-Weitzenb\"{o}k formula on the space of one forms we have,
$$\la df=D^*Ddf+(n-1)cdf$$
Hence the lemma follows.
\end{proof}
\begin{lem}$\langle\Rt '_g(fg),fg\rangle= 4\frac{s}{n}\|df\|^2-|R^2|\|f\|^2 $
\end{lem}
\begin{proof}
\Bea \Rt _{pq}=\gt^{i_1i_2}\gt^{j_1j_2}\gt^{k_1k_2}R_{pi_1j_1k_1}R_{qi_2j_2k_2}
\Eea
Differentiating it with respect to $t$ and evaluating on an orthonormal basis we have,
\Bea (\Rt_g.h)'_{pq}&=&-h_{mn}\left(R_{pmij}R_{qnij}+R_{pimj}R_{qinj}+R_{pijm}R_{qijn}\right)\\
&&+(R'_g.h)_{pijk}R_{qijk}+R_{pijk}(R'_g.h)_{qijk}
\Eea
Therefore,
\Bea\langle\Rt '_g(fg),fg\rangle &=& -3|R|^2\|f\|^2+ 2\langle (R'_g.fg), fR\rangle \Eea
Using (2.7) we have,
\bea (\Rt '_g(fg),R)&=&\frac{1}{2}\sum[Ddf_{jk}R_{ijki}+Ddf_{il}R_{ijjl}-Ddf_{ik}R_{ijkj}-Ddf_{jl}R_{ijil}]+f|R|^2\\
\nonumber&=&2\mu\la f+f|R|^2
\eea
Hence the lemma follows.
\end{proof}
\begin{lem}$(|R|^p)'(fg)=2\frac{s}{n} p|R|^{p-2}\la f -pf|R|^p$
\end{lem}
\begin{proof}
\Bea (|R|^p)'(fg)&=&p|R|^{p-2}( R,R'_g.fg)-2p|R|^{p-2}(\Rt,fg)\\
&=&p|R|^{p-2}(R,R'_g.fg)-2pf|R|^{p-2}tr(\Rt)\\
&=&p|R|^{p-2}(2\mu \la f+f|R|^2)-2pf|R|^p\\
&=&2\mu p|R|^{p-2}\la f -pf|R|^p
\Eea
\end{proof}
Using (2.6) we have,
$$(Dd^Dfg,R)=2(Ddf,r)=-2\mu \la f$$
Now the Proposition follows from the above lemma and equation. \\
\\
{\it Proof of Theorem 1.1:} Let $(M,g)$ is a simply connected irreducible symmetric space of compact type which is not a sphere. Then
$$R=\frac{s}{n(n-1)}I+W$$
where $I$ is the curvature of standard sphere with sectional curvature $1$ and $W$ is the Weyl curvature of $(M,g).$ From the above expression we have
$$\frac{s^2}{|R|^2}< \frac{2}{n(n-1)}$$
Let $\lambda_1$ be the first positive eigenvalue of the Laplacian of $(M,g)$ and $f$ be an eigenfunction corresponding to $\lambda_1$. Then from  Proposition 2.1 we have,
\Bea H(fg,fg)&=& s\lambda_1p|R|^{p-2}[a(\frac{\lambda_1}{s})-\frac{b}{s}]\|f\|^2\\
&\leq& s\lambda_1p|R|^{p-2}[\huge(n-1 + 4\frac{n-4}{n^3(n-1)}\huge)\frac{\lambda_1}{s}-2\frac{n-2}{n}]\|f\|^2
\Eea
From the Table A.1 and A.2 in \cite{HU} we have,
$$H(fg,fg)<0.$$
Next, we choose a sufficiently large eigenvalue $\lambda_i$ such that $a(\frac{\lambda_i}{s})-\frac{b}{s}<0$. Let $\tilde{f}$ be an eigenfunction corresponding to $\lambda_i.$ Then we have, 
$$H(\tilde{f}g,\tilde{f}g)<0$$
This completes the proof.
\hfill
$\square$\\
\begin{thm2} Let $(M,g)$ be either a compact quotient of an irreducible symmetric of non-compact type or a compact symmetric space which is not one of the types in Theorem 1. $(M,g)$ is stable for $\mathcal{R}_p$ for $p\geq \frac{n}{2}$ when it is restricted to the space of conformal variations of $g$.
\end{thm2}
\begin{proof} If $(M,g)$ is a a compact quotient of an irreducible symmetric of non-compact type then the theorem is an immediate consequence of Proposition 2.1. Otherwise from the Table A.1 and A.2 in \cite{HU} we have $\frac{\lambda_1}{s}\geq \frac{2}{n}.$ Therefore,
\Bea H(fg,fg)&\geq& sp\lambda_1|R|^{p-2}[a(\frac{\lambda_1}{s})-\frac{b}{s}]\|f\|^2\\
&\geq& sp\lambda_1|R|^{p-2}[(n-1)\frac{\lambda_1}{s}-2\frac{n-2}{n}]\|f\|^2\\
&\geq& 2sp\lambda_1|R|^{p-2}[\frac{n-1}{n}-2\frac{n-2}{n}]\|f\|^2\\
&\geq& 0
\Eea
\end{proof}
{\bf Remark :} Let $(M,g)$ be one of the critical metrics of $\mathcal{R}_p$ mentioned in Theorem 2. This is an immediate consequence of the above theorem that $(M,g)$ is a local minimizer for $\mathcal{R}_p(p\geq \frac{n}{2})$ restricted to the space of metrics conformal to $g$.

\end{document}